\documentclass[12pt]{amsart}
\usepackage[margin=1in]{geometry}
\usepackage{mathptmx}
\usepackage{mathtools}
\usepackage{amsthm}
\usepackage{amssymb}
\usepackage[alphabetic]{amsrefs}
\usepackage{graphicx}
\usepackage{tikz}
	\usetikzlibrary{decorations.pathreplacing}
	\usetikzlibrary{decorations.pathmorphing}
	\usetikzlibrary{patterns}
\usepackage{caption}

\newcommand{\cd}{\cdot}
\newcommand{\ra}{\rightarrow}
\newcommand{\pr}{\prime}

\newcommand{\R}{\mathbb{R}}
\newcommand{\Z}{\mathbb{Z}}

\newcommand{\abs}[1]{\left\lvert #1 \right\rvert}

\newcommand{\lbar}[1]{\overline{#1}}

\DeclareMathOperator{\dist}{dist}

\renewcommand{\coprod}{\rotatebox[origin = c]{180}{$\prod$}}

\newtheorem{thm}{Theorem}

\newtheorem{lemma}[thm]{Lemma}

\theoremstyle{definition}

\newtheorem*{defin*}{Definition}
\theoremstyle{remark}
\newtheorem*{note}{Note}

\begin{document}

\title{Packing Meets Topology}

\author{Michael H. Freedman}
\address{\hskip-\parindent Michael H. Freedman, Center of Mathematical Sciences and Applications, Harvard University}
\date{\today}

\begin{abstract}
	This note initiates an investigation of packing links into a region of Euclidean space to achieve a maximal density subject to geometric constraints. The upper bounds obtained apply only to the class of homotopically essential links and even there seem extravagantly large, leaving much working room for the interested reader.
\end{abstract}

\maketitle

\section{Introduction and Theorems}
Optimal packing of balls into Euclidean space has a long history and recent astonishing successes, including Hale's resolution of the Kepler conjecture \cite{hales17} and the optimality of the $E_8$ and Leech lattices \cites{viazovska17,CKMRV17}, resulting in a 2022 Fields Medal to Maryna Viazovska.

In this note, we introduce the idea of packing links, rather than points, again with the goal of achieving the highest possible density subject to the geometric constraints that \emph{certain} link components must maintain a distance $\geq \epsilon$ from \emph{certain} other components. There will be an observation about higher dimensions but let us begin with packing classical links into Euclidean 3-space. In the classical sphere packing problem, \emph{all} points are constrained to have distance $\geq \epsilon$ from each other. The analogous stipulation for links, that all components must maintain a distance $\geq \epsilon$ from each other, is also of potential interest, but in that case, the coarse outline of the subject is broadly similar to point packings. That is, in both cases, each component takes up a definite amount of volume so only $O(\epsilon^{-3})$ link components can be $\epsilon$-embedded into the unit cube, where $\epsilon$-embedded means no two components approach within $\epsilon$ of each other.\footnote{We use $(O,\Omega,\Theta)$-notation, $g(x) = O(f(x))$, to mean that, for some $a > 0$ and sufficiently large $x$, $g(x) \leq af(x)$, $g(x) = \Omega(f(x))$ for the opposite inequality, and $g(x) = \Theta(f(x))$ if both inequalities hold but for different constants.} While this coarse upper bound holds for all link types, complicated links almost surely have smaller upper bounds. For example, we conjecture that if $L_n$ is the link type consisting of $n$-fibers of the Hopf map $S^3 \ra S^2$ then $n$ can grow no more quickly than $O(\epsilon^{-2})$.

However, this note focuses on a regime, \emph{partial-$\epsilon$-embeddings}, where even the coarse answer can be quite mysterious. By a partial-$\epsilon$-embedding we mean that only certain specified pairs of components must stay $\epsilon$ apart. In this context we are often left puzzled as to whether the number of link components that can fit into the unit cube is: (1) countably infinite\footnote{This case, of course, would require slightly relaxing the definition of ``embedding'' to a 1-1 map which, when restricted to any finite collection of circles, is a smooth embedding of the expected link type.}, (2) finite but unbounded, (3) exponential or super-exponential in $\epsilon^{-1}$, or (4) polynomial in $\epsilon^{-1}$.

The theme of this note is well illustrated by our first example: $H^n$, which by definition is the $2n$-component link type formed by taking $n$-Hopf links $\vphantom{l}_r$\tikz{\draw[thick] (0.09,0.12) arc (55:375:0.15);\draw[thick] (0.09,-0.04) arc (195:-120:0.15);}$\vphantom{l}_b$ each separated from the others by some smooth embedded 2-sphere. One may write the link as $H^n \coloneqq \coprod_{j=1}^n H_j$, where $H_j = r_j \cup b_j$. The partial-$\epsilon$-embedding condition we study is that for all $j$, $\dist(r_j, b_j) \geq \epsilon$. We call such an embedding a \emph{diagonal-$\epsilon$-embedding}.

\begin{thm}\label{thm:diagonal-embed}
	If $H^n$ has a diagonal-$\epsilon$-embedding into the unit cube then $n(\epsilon) = O\left(e^{a\epsilon^{-3}}\right)$ for some $a > 0$.
\end{thm}

\begin{proof}
	Let $I^3$ be the unit cube. Tile $I^3$ by cells dual to a triangulation of $I^3$. The cells should have the property that they are somewhat regular: each cell should have an inscribed sphere of radius $> \frac{\epsilon}{20}$ and an excribed sphere of radius $< \frac{\epsilon}{2}$. These dual cells have the property that any union of them is a PL 3-manifold with boundary. We prefer not to use the obvious coordinate sub-cubes of $I^3$, because they fail to have this property.\footnote{For this first proof, the manifold property is actually not necessary, but for Theorems \ref{thm:finite-int} and \ref{thm:n-max} the manifold property is an added convenience.} The number of cells in this tiling $\tau$ is $O(\epsilon^{-3})$.

	Assume $H^n$ is diagonally-$\epsilon$-embedded. For each $j$, $1 \leq j \leq n$, 3-color the tiling $\tau$ according to the rule that a cell is red if it meets $r_j$, blue if it meets $b_j$, and white otherwise. Call this coloring $c_j$. Now we decorate $c_j$ with additional homological information. Let $R_j$ ($B_j$) be the union of the red (blue) cells under $c_j$. The first homology, $H_1(R_j; \Z_2)$ is a vector space over $\Z_2$ of dimension $d_j = O(\epsilon^{-3})$, for which we choose a basis $f_{j1}, \dots, f_{jd_j}$. Similarly, $H_1(B_j; \Z_2)$ has dimension $e_j = O(\epsilon^{-3})$ with basis $g_{j1}, \dots, g_{je_j}$. Let $L_{jpq}$ be the $d_j \times e_j$ matrix of $\Z_2$-linking numbers, $L_j^{pq} = \operatorname{Link}(f_{jp},g_{jq})$.

	Now, homologically we may express the class of $r_j$ in $H_1(R_j)$ (the class of $b_j$ in $H_1(B_j)$) as $r_j = \sum_{p=1}^{d_j} x_{jp}f_{jp}$ ($b_j = \sum_{q=1}^{e_j} y_{jq}g_{jq}$). Now the mod 2 linking numbers of the link $H_j$ can be recovered as:
	\begin{equation}\label{eq:link-numbers}
		1 = \operatorname{Link}(b_j, r_j) = L_j^{pq}x_{jp}y_{jq},
	\end{equation}
	Einstein summation convention in effect.

	Where $c_j$ was the $j^{\text{th}}$ coloring, let $\hat{c}_j$ be a decorated $j$-coloring where the decoration amounts to fixing the mod 2 numbers $\{x_p\}$, $1 \leq p \leq d_j$, and $\{y_q\}$, $1 \leq q \leq e_j$, which express, within the arbitrarily chosen bases, how $b_j$ and $r_j$ lie homologically in $H_1(B_j; \Z_2)$ and $H_1(R_j; \Z_2)$.

	How many possible decorated colorings, $\#_{\text{DC}}(\epsilon)$ can there be?
	\begin{equation}
		\#_{\text{DC}}(\epsilon) = O\left(3^{a^\pr \epsilon^{-3}} \cd 2^{a^{\pr\pr} \epsilon^{-3}} \cd 2^{a^{\pr\pr} \epsilon^{-3}}\right) = O\left(e^{a \epsilon^{-3}}\right),
	\end{equation}
	all constants $> 0$.

	The first factor bounds the number of 3-colorings and the second two factors the possible values of the binary strings $\{x_{jp}\}$ and $\{y_{jq}\}$, respectively.

	Now by the pigeonhole principle if $n(\epsilon)$ were \emph{not} $O(e^{a\epsilon^{-3}})$, two Hopf links $H_i$ and $H_j$, $i \neq j$, within $H^n$ must determine the \emph{same} decorated coloring $\hat{c}_i = \hat{c}_j$. But with $H_i$ and $H_j$ having identical thickenings: $B_i = B_j$ and $R_i = R_j$, and identical homological data. Line \ref{eq:link-numbers} can also be read as a computation for the \emph{off-diagonal} linking number:
	\begin{equation}
		0 = \operatorname{Link}(b_j, r_i) = L_j^{pq}x_{jp}y_{iq} = L_i^{pq}x_{jp}y_{jq} = L_i^{pq}x_{ip}y_{iq} = 1
	\end{equation}

	This contradiction proves the theorem.
\end{proof}

Before leaving this example, what packings can we imagine to supply a lower bound on $n(\epsilon)$, for $H^n$? The simplest starting point would be to link two circles of radius $\epsilon$ into a small, rigid, Hopf link, and then throw copies of these ``linked key rings'' into a unit box, shaking gently until full. This seems to yield $n = \Theta(\epsilon^{-3})$. But then we realize the box is not as full as we thought. We can sprinkle in a second generation of orthogonally linked pairs of radius $3\epsilon$ circles, ignoring the presence of the first generation. By ignoring the first generation, we will create many linking number $=1$ with the first generation, but these can be undone ``finger moves'' of length $\leq \epsilon$ to the second generation. By the triangle inequality, the second generation will still satisfy the diagonal-$\epsilon$-embedded condition after all finger moves. We are still not done; we can add a third generation of orthogonal radius $= 7 \epsilon$ Hopf links, which will retain the diagonal-$\epsilon$ condition after length $\leq 3$ finger moves recovers the correct link type, $H_n$. We can of course itterate with Hopf links of radius $\{r_i\}$, $r_0 = \epsilon$, $r_{i+1} = 2r_i + 1$, until $r_i$ approaches unit size. From scale considerations, but ignoring unimportant boundary effects, we see that if $n_i$ is the number of $i^{\text{th}}$ generation Hopf links in the box, then $n_0, n_1, n_2, \dots$ is dominated by the geometric series $n_0, 2^{-3}n_0, 4^{-3}n_0, \dots$. So summing this series we find that the total number $n = \sum n_i$ of the Hopf links satisfies
\begin{equation}
	n(\epsilon) < \frac{8}{7} n_0
\end{equation}

So, in the end, all our extra work only changed (slightly) the leading coefficient. Not being able to find anything more clever, this leaves the huge gap between $\Omega(\epsilon^{-3})$ and $\Omega(e^{a\epsilon^{-3}})$, in which the truth must lie. Our conjecture is that $n = \Omega(\epsilon^{-3})$, but the proof calls out for a new idea.

Before discussing other link types, let us make a quick remark regarding higher dimensions. If $d = p+q+1$, then two disjoint closed submanifolds of $\R^d$ have a well-defined mod 2 linking number if they have dimensions $p$ and $q$ respectively. Now in $\R^d$ let $H^n$ denote any link of $2n$ component $\{r_j, b_j\}$, $1 \leq j \leq n$, with mod 2 linking numbers given by:
\[
	L(r_i,r_j) = 0,\ L(b_i, b_j) = 0,\ i \neq j, \text{ and } L(r_i,b_j) = \delta_{ij}
\]

Identical reasoning shows that the maximum possible $n(\epsilon)$, $n_{\text{max}}(\epsilon)$, satisfies:
\begin{equation}
	O(\epsilon^{-d}) \leq n_{\text{max}}(\epsilon) \leq e^{a(d) \epsilon^{-d}}
\end{equation}
for some $a > 0$, which actually generalizes Theorem \ref{thm:diagonal-embed} even when $d = 3$. $n_{\text{max}}(\epsilon)$ is the largest number such that a $2n(\epsilon)$-component link can be embedded in the unit $d$-cube with the specified linking and $\dist(r_j,b_j) \geq \epsilon$, $1 \leq j \leq n$.

Returning to dimension $d=3$, let us give a further example, which steps away, slightly, from linking number. Consider the problem of packing the disjoint union (again this means smoothly embedded spheres separating the copies of) of $n$ copies $B^n$ of a three component link $B$, such as the Borromean rings, which has all linking numbers 0 and Milnor's $\lbar{\mu}$-invariant $\lbar{\mu}_{123}(L) \not\equiv 0$ mod 3 \cite{milnor54}. $B$ has components $l_1, l_2, l_3$, $B_j = (l_{j1},l_{j2},l_{j3})$. Again, colors $r_1,r_2,r_3$ are associated to the 3-components. We now enforce the diagonal-$\epsilon$-condition: for each $j$, $1 \leq j \leq n$, $\dist(l_{ji},l_{ji^\pr}) \geq \epsilon$ whenever $i \neq i^\pr$.

Let $n_B(\epsilon)$ be the largest $n$ for which such an embedding exists, or $\infty$ is no such bound exists.

\begin{thm}\label{thm:finite-int}
	For all $\epsilon > 0$, the Borromean packing number $n_B(\epsilon)$ is indeed a finite integer, with $n_B(\epsilon) = O(e^{a\epsilon^{-9}})$.
\end{thm}

\begin{proof}
	We begin, as before, with a generic tessellation of $I^3$ of scale between $\frac{\epsilon}{20}$ and $\frac{\epsilon}{2}$. Now, for each $j$, $1 \leq j \leq n$, make a 4-coloring $c_j$ of $I^3$ by the rule that a cell gets the color $r_i$, $1 \leq i \leq 3$, of the component it meets; if it meets none then it is white. But now we proceed differently, for the \emph{decoration}: homology is wholly insufficient. To motivate our new decoration recall a classic:

	\begin{thm}[Burnside]
		Any finitely generated group of exponent 3, meaning every element has order $\leq 3$, is finite.
	\end{thm}
\end{proof}

\begin{note}
	To estimate $n_B(\epsilon)$ in Theorem \ref{thm:finite-int}, we used the calculation of \cite{LB33} that the order of the free, restricted Burnside group is $\abs{B(m,3)} = 3^{m + \binom{m}{2} + \binom{m}{3}}$. This will imply the bound stated in Theorem \ref{thm:finite-int}.
\end{note}

We create a bespoke invariant to exploit Burnside's theorem.

\begin{defin*}
	Define \emph{3-link-homotopy} to be Milnor's classical link-homotopy \cite{milnor54} (individual components may cross themselves during the homotopy but not other components) with the additional ad hoc relation: at any moment during the homotopy, any component may be band summed to $g^3$, where $g$ is a free loop in the complement of the other components. The cube means wrap 3 times around $g$.
\end{defin*}

Whereas before, the coloring $c_j$ was decorated with homological information, now the decoration $\hat{c}_j$ assigns to the submanifold $C_{ji}$ colored $r_i$ (according to our rule for the $j^{\text{th}}$ coloring $c_j$) the conjugacy class $[l_{ji}]$ of the component $l_{ji}$ in the Burnside group $\pi_1^3(C_{ji})$, where by definition, $\pi_1^3(X)$ means $\pi_1(X)$ with the additional relations that all elements cube to the identity.

\begin{lemma}\label{lm:3-link}
	The 3-link-homotopy class of a link $B_j$ in $I^3$ can be recovered from the decorated coloring $\hat{c}_j$.
\end{lemma}

\begin{proof}
	Since $C_{ji}$ and $C_{ji^\pr}$ are disjoint for $i \neq i^\pr$, a homotopy of $L_j$ in which each component $l_{ji}$ stays within its $C_{ji^\pr}$ is a link-homotopy. Furthermore, if each $l_{ji}$ is permitted to vary in $C_{ji}$ within its $\pi_1^3(C_{ji})$ conjugacy class, this is a special case of 3-link-homotopy. Thus, if each $l_{ji}$ is rechosen within its $\pi_1^3(C_{ji})$ conjugacy class, the 3-link homotopy class is preserved.
\end{proof}

\begin{lemma}\label{lm:vanishing-ln}
	For a 3-component link with vanishing linking numbers\footnote{It is actually only necessary to assume $3 \nmid \operatorname{gcd}(\operatorname{link}(l_i,l_j))$, $i \neq j$.}, $\lbar{\mu}_{123}(B)$ is conserved mod 3 under 3-link-homotopy.
\end{lemma}

\begin{proof}
	We may assume that during the 3-link-homotopy only one component moves or is altered at any given time. The ``cyclic symmetry'' theorem (\cite{milnor57} Theorem 6) says that w.l.o.g.\ we may assume that the active component is the one being Magnus-expanded in the link group of the others. To recall, for any $k$-component link $L$, $\lbar{\mu}_I(L)$, $I = i_1,\dots,i_k$ distinct indices, is computed by expending, as below, the component $l_{i_k}$ in the polynomial ring denoted by $R[x_{i_1},\dots,x_{i_{k-1}}]$ \cite{milnor54}. This is Milnor's notation for the integers adjoined $k-1$ non-commuting variables which are also ``non-repeating,'' meaning that one divides out by the ideal generated by monomials in which any variable occurs more than once.

	\begin{equation}\label{eq:magnus-expand}
		\begin{tikzpicture}[baseline={(current bounding box.center)}]
			\node at (0,0.1) {$[l_{i_k}] \in \pi_1(I^3 \setminus (l_{i_1} \cup \cdots \cup l_{i_{k-1}}))$};
			\node[rotate around = {-90:(0,0)}] at (0,-0.4) {$\twoheadrightarrow$};
			\node at (0.55,-0.9) {$M(I^3 \setminus (l_{i_1} \cup \cdots \cup l_{i_{k-1}}))$};
			\node[rotate around = {90:(0,0)}] at (0,-1.4) {$\twoheadrightarrow$};
			\node at (2.5,-1.8) {$FM_{k-1}(m_{i-1},\dots,m_{i_{k-1}}) \xrightarrow{\text{Magnus}} R[x_{i_1},\dots,x_{i_{k-1}}]$};
			\node at (0.9,-2.5) {$m_{i_j} \mapsto 1 + x_i$};
			\node at (1,-3.05) {$m_{i_j}^{-1} \mapsto 1 - x_i$};
		\end{tikzpicture}
	\end{equation}
	where $m_{i_j}$ are meridians to $l_{i_j}$, $M$ denotes the Milnor link group obtained by adding the relations that each meridian commutes with all its conjugates, and FM is the corresponding free Milnor group generated by $m_{i_1},\dots,m_{i_{k-1}}$ subject only to these commutation relations. As the diagram indicates, $[l_{i_k}]$ is first projected, then lifted to $FM_{k-1}$, and finally expanded.

	Then by definition, $\lbar{\mu}_I(L) =$ the coefficient of $x_{i_1},\dots,x_{i_{k-1}}$ of Magnus$[l_{jk}]$. Any ambiguity in the expansion due to the choice of lifting constitutes the indeterminancy of that $\lbar{\mu}_I$. For general background on $\lbar{\mu}$ invariants see \cites{milnor54,milnor57,krushkal98}.

	As the $k^{\text{th}}$-component moves by link-homotopy the element and its expansion are constant. Adding the cube $g^3$ of a loop $g$ to $l_{jk}$ multiplies its Magnus expansion $M$ by the Magnus expansion $M_{g^3}$ of $g^3$, $M \ra MM_{g^3} = M(M_g)^3$.

	Since $B$ has 3 components, $k - 1 = 2$, $(M_g)^3$ is the cube of some monic polynomial in two variables $x_1$ and $x_2$:
	\[
		(M_g)^3 = (1 + c_1 x_1 + c_2 x_2 + c_{12}x_1x_2 + c_{21}x_2x_1)^3
	\]
	Thus, the coefficients of $x_1$, $x_2$, $x_1x_2$, and $x_2x_1$ in $(M_g)^3$ are all divisible by 3. Multiplying out we see that $\lbar{\mu}_{123}(B)$ mod 3 is invariant under 3-link-homotopy.
\end{proof}

The number $\#_c(\epsilon)$ of possible colors is $\#_c(\epsilon) = O(e^{a^\pr \epsilon^{-3}})$ and the number of decorations possible for a coloring $c$ is bounded by the product of the order of the Burnside groups $\Omega(e^{a^{\pr\pr}\epsilon^{-9}})$ for each of the colored (not white) regions. Thus, the number of possible decorated coloring $\#_{\hat{c}(\epsilon)} = \abs{\{\hat{c}(\epsilon)\}}$ has a similar bound as a function of $\epsilon$. As in Theorem \ref{thm:diagonal-embed}, the pigeonhole principle tells us that if we could place $n > \#_{\hat{c}(\epsilon)}$ copies of $B$ in $I^3$, obeying the diagonal-$\epsilon$-condition then for $1 \leq i < j \leq n$, then $B_i$ and $B_j$ will determine identical decorated colorings.

But Lemma \ref{lm:3-link} now tells us three things: $L_i$ has 3-link-homotopy type $B$, $B_j$ has 3-link-homotopy type $B$, and $B_{ij}$ has 3-link-homotopy type $B$, where $B_{ij}$ is the link obtained by starting with $B_i$ and then swapping out any one component of $B_i$ for the corresponding component of $B_j$. The first two conclusions are as we expect, but the third sounds wrong. Because $B_i$ and $B_j$ are split (separated by a smoothly embedded 2-sphere), so $B_{ij}$ is a split link and all its $\lbar{\mu}_{123}$-invariant must vanish. But this vanishing contradicts Lemma \ref{lm:vanishing-ln}, which says any link (including $B_{ij}$) in the 3-link-homotopy class of $B$ has its $\lbar{\mu}_{123}$-invariant \emph{not} congruent to 0 mod 3. This proves Theorem \ref{thm:finite-int}. \qed

Replacing Burnside groups with the mod $p$ lower central series ($p$-lcs) quotients allows a joint extension of Theorems \ref{thm:diagonal-embed} and \ref{thm:finite-int}, although with a weaker (larger) upper bound.

\begin{thm}\label{thm:n-max}
	Let $E$ be any homotopically essential link of $k$-components, $E = (e_1,\dots,e_k)$ and $E^n$ be the disjoint union of $n$ copies of $E$. For every $\epsilon > 0$ there is a largest $n$, $n_{\text{max}}$, such that $E^n$ embeds in the unit cube $I^3$ with the property that for all $j$, $1 \leq j \leq n$, $\dist(e_{ji}, e_{ji^\pr}) \geq \epsilon$ for $1 \leq i \neq i^\pr \leq k$. $n_{\text{max}} = O\left((k+1)^{a\epsilon^{-3}}p^{{(a\epsilon^{-3})}^{k-1}}\right)$, where $p$ is the smallest prime \emph{not} dividing any lowest order, nontrivial, non-repeating $\lbar{\mu}$-invariant of $E$, and $a > 0$ a fixed constant.
\end{thm}

\begin{proof}
	Begin in the familiar fashion by creating a $(k+1)$-coloring the cells $c_j$ of a fixed $\epsilon$-scale tessellation of $I^3$ in which each cell meeting $e_{ji}$ is colored $r_i$ and the remaining cells are colored white. Similar to Theorems \ref{thm:diagonal-embed} and \ref{thm:finite-int}, we need to specify some finite amount of data about $e_{ji}$ in its $r_i$-colored region $C_{ji}$ sufficient to (1) certify the homotopically essential nature of $E_j$ and (2) create the contradiction that some related split link $E_{jj^\pr}$, defined below, would also be homotopically essential.

	By induction, it suffices to consider the case that $E$ is almost homotopically trivial, meaning all its sub-links are all homotopically trivial, or more algebraically, that all non-repeating $\lbar{\mu}$-invariants of length $<k$ vanish.

	As in the proof of Theorem \ref{thm:finite-int}, cyclic symmetry implies that we may focus on a single ``active'' component $e_{ij_\ell}$ (and going forward drop the $j$-index for the embedding and replace $i_\ell$ by a single index $k$), project $e_{ij_\ell}$, now denoted simply by $e_k$, to the nilpotent $p$-group $\pi_1(C_{ji}) \slash [\pi_1(C_{ji})]_k^p$, where for any group $G$, $[G]_n^p$ is the $n^{\text{th}}$-term of the mod $p$ lower central series of $G$. This is defined by saying $G_1 = G$, and $G_m$ is generated ($=$ normally generated) by the words $a u a^{-1} u^{-1} v^p$, $a \in G$, and $u,v \in G_{m-1}$. It is the order of this group that enters the bound, but for link homotopy computations, it suffices to lift $e_k$ to $\alpha_k$ in the corresponding quotient of the free Milnor group, $\alpha_k \in Q_k^p := FM_{k-1} \slash [FM_{k-1}]_k^p$. The notation $e_k$ intentionally suggests that it is the last component which is ``active.'' It is harmless to assume this given almost triviality and cyclic symmetry.

	Regarding the bound, its essential ingredient is that the order $\abs{Q_k^p} = O\left(p^{\left((a^\pr \epsilon^{-3})^{(k-1)}\right)}\right)$. The group $\pi_1(C_{ji})$ has $g = O(\epsilon^{-3})$ generators, so this also bounds the number of generators of the free Milnor group under consideration. The quotient $Q_k^p$ is $(k-1)$-stage nilpotent with at most $g^s$ new (twisted) $\Z p$ factors added by during the $s^{\text{th}}$-central extension. Thus, the total number of copies of $\Z_p$ twisted together to make the $p$-group $Q_{k-1}^p$ is $O(\epsilon^{-3})^{k-1}$, giving the order bound.

	Returning to the main line of the proof, we need:

	\begin{lemma}\label{lm:fm}
		Suppose $\beta \in [FM(m_1,\dots,m_{k-1})]_i^p$, $1 \leq i \leq k-1$, then $\operatorname{Magnus}(\beta)$ maps to $(1 + \textup{monomials of degree} \geq i)$ under reduction of coefficients $\Z \ra \Z_p$, inducing $R[x_1,\dots,x_{k-1}]$ $\xrightarrow{\pi_p} R_p[x_1,\dots,x_{k-1}] \coloneqq \Z_p[x_1,\dots,x_{k-1}] \slash (\textup{repeating ideal})$, i.e.\ $\pi_p(\operatorname{Magnus}(\beta)) = (1 + \textup{terms of degree}$ $\geq i)$.
	\end{lemma}

	\begin{proof}
		By induction. When $i=1$ the statement is that $p^\text{th}$ power has no linear terms when expanded into $R_p[x_1,\dots,x_{k-1}]$. Now assume that Lemma \ref{lm:fm} is true for $i-1$ and expand $a u a^{-1} u^{-1} v^p$, where $a \in FM(m_1,\dots,m_{k-1})$, and $u,v \in [FM(m_1, \dots, m_{k-1})]_{i-1}^p$. The lowest positive degree ($= i-1$) monomials in Magnus$(u)$ and Magnus$(u^{-1})$ are identical except for reversed signs; consequently, the $a u a^{-1} u^{-1}$ factor expands to $(1 + \text{monomials of degree} \geq i)$ as the degree $= i-1$ terms all cancel. The $v^p$ factor has the same form since the degree $i-1$ terms are now repeated $p$ times each. Consequently, the product $a u a^{-1} u^{-1} v^p$ also expands to this form.
	\end{proof}

	The mod $p$-lcs subgroups are characteristic: they map to each other under homomorphisms and if $F \twoheadrightarrow G$ is an epimorphism then $[F]_k^p$ maps epimorphically to $[G]_k^p$. Apply these facts to the maps:
	\begin{equation}
		\begin{tikzpicture}[baseline={(current bounding box.center)}]
			\node at (0,0) {$\pi_1(C_k) \ra \pi_1(I^3 \setminus e_1 \cup \cdots \cup e_{k-1}) \ra M(I^3 \setminus e_1 \cup \cdots \cup e_{k-1}) \leftarrow FM(I^3 \setminus e_1 \cup \cdots \cup e_{k-1})$};
			\node at (6.4,-0.6) {\footnotesize{Magnus}};
			\node[rotate=-90] at (5.7,-0.6) {$\longrightarrow$};
			\node at (5.7,-1.1) {$R_p[x_1,\dots,x_{k-1}]$};
			\node[rotate=90] at (-6.5,-0.65) {$\longrightarrow$};
			\node at (-6.7,-1.3) {$\beta \in [\pi_1(C_k)]_k^p$};
		\end{tikzpicture}
	\end{equation}
	and apply Lemma \ref{lm:fm} to conclude that any $\beta \in [\pi_1(C_k)]_{k-1}^p$ will Magnus expand to 1 in line (\ref{eq:magnus-expand}).

	Now consider a second version of a bespoke link homotopy, $(p,k)$-link-homotopy, in which components homotope (while maintaining disjointness), for convenience only move one at a time, and finally the active component (which our notation treats as the last component) is permitted at any moment to form an ambient connected sum with any loop $\beta \in [\pi_1(C_k)]_k^p$. The present analog of Lemma \ref{lm:vanishing-ln} is that the non-repeating length $k$ $\lbar{\mu}$-invariants of $L$ are invariant mod $p$ under $(p,k)$-link-homotopy. The proof is parallel to that of Lemma \ref{lm:vanishing-ln}, simply Magnus expand $e_k \# \beta \subset I^3 \setminus (e_1 \cup, \cdots \cup e_{k-1})$ into $R_p[x_1, \dots, x_{k-1}]$.

	The present analog of Lemma \ref{lm:3-link} is that any $k$-component $E^\pr$ $(p,k)$-link-homotopic to $E$ continues to have nontrivial, non-repeating, $\lbar{\mu}$ invariants of length $k$. This follows from Lemma \ref{lm:fm}, again by expanding $e_k \# \beta$ into $R_p[x_1, \dots, x_{k-1}]$. In particular, no such $E^\pr$ can be a split link.

	The proof of Theorem \ref{thm:n-max} is completed, once again, by an application of the pigeonhole principle. If $n_{\text{max}}(\epsilon)$ exceeds the cardinality of the decorate colorings $\{\hat{c}_j\}$, where now each colored region $C_{ji}$ is decorated by a conjugacy class of $[\pi_1(C_{ji})]_k^p$ representing the invariant information regarding the location of $e_{ji}$ inside $C_{ji}$, then for $j \neq j^\pr$, $E_j$ and $E_{j^\pr}$ will induce identical data. The number of decorated colorings is bounded by the number of possible colorings $\abs{\{c_j\}} = O\left((k+1)^{a \epsilon^{-3}}\right)$ times $\abs{Q^p_k}$ which bounds the choice of decoration, so $\abs{\{\hat{c}_j\}} = O\left((k+1)^{a\epsilon^{-3}}p^{{(a\epsilon^{-3})}^{k-1}}\right)$. We have just argued that this data suffices to reconstruct the nontrivial $(p,k)$-link-homotopy classes of both $E_j$ and $E_{j^\pr}$. This is as it should be. But now define $E_{jj^\pr}$ by starting with $E_j$ and exchanging any one of its components with the corresponding component of $E_{j^\pr}$. Exactly the same data now tells us that $E_{jj^\pr}$ has a non-vanishing, $\lbar{\mu}$-invariant of length $k$. This is a contradiction since $E_{jj^\pr}$ is a split link, split by the 2-sphere separating $E_j$ from $E_{j^\pr}$, so $\lbar{\mu}$-invariants of length $k$ vanish.
\end{proof}

\section{Discussion}
The use of $\Z_2$ coefficients in the initial homological disucssion was arbitrary, and any finite coefficient ring would suffice. However, for Theorem \ref{thm:finite-int}, the choice of the prime 3 was crucial. $p = 2$ would make the Burnside group abelian and provide no useful information. In this regard it is amusing to check that the Borromean rings is indeed 2-link-homotopy equivalent to the 3-component unlink ($\lbar{\mu}_{123}$ is \emph{not} conserved mod 2 under 2-link-homotopy). The restricted Burnside groups $B(n,k)$ are only known to be finite for $k = 2,3,4,$ and 6. The most general Theorem \ref{thm:n-max} exploits the interplay of the mod $p$-lcs with the $\lbar{\mu}$-invariants. While broadest, the estimate there is somewhat worse.

Our philosophy is that the upper bounds we offer, based on homology or $\lbar{\mu}$, are \emph{terrible}. Firstly, the estimates seem way too big, and second they only apply to links with easy algebraic features; boundary link and even the Whitehead link are left untouched. Our conjecture, a challenge to the reader, is that every non-trivial link $L$ of two or more components has an $\epsilon$-diagonal packing bound for the number of $\epsilon$-diagonally embedded copies of the form $\#_L(\epsilon) = O(\epsilon^{-3})$.

\section{Acknowledgements}
The question studied here arose while working with Michael Starbird on \cite{fs22}. An $O(\epsilon^{-3})$ bound for the Hopf link problem might offer an alternative proof strategy for that paper's main theorem. I would also like to thank Slava Krushkal for insightful discussions, and the referee for important clarifications.

\bibliography{references}

\begin{bibdiv}
\begin{biblist}

\bib{CKMRV17}{article}{
      author={Cohn, Henry},
      author={Kumar, Abhinav},
      author={Miller, Stephen},
      author={Radchenko, Danylo},
      author={Viazovska, Maryna},
       title={The sphere packing problem in dimension 24},
        date={2017},
     journal={Ann. Math.},
      volume={185},
      number={3},
       pages={1017\ndash 1033},
}

\bib{fs22}{article}{
      author={Freedman, Michael},
      author={Starbird, Michael},
       title={The geometry of the {B}ing involution},
        date={2022},
      eprint={arXiv:2209.07597},
}

\bib{hales17}{article}{
      author={Hales, Thomas},
      author={others},
       title={A formal proof of the {K}epler conjecture},
        date={2017},
     journal={Forum Math. Pi},
      volume={5},
}

\bib{krushkal98}{article}{
      author={Krushkal, Vyacheslav},
       title={Additivity properties of {M}ilnor's $\lbar{\mu}$-invariants},
        date={1998},
     journal={J. Knot Theory Ramif.},
      volume={7},
      number={5},
       pages={625\ndash 637},
}

\bib{LB33}{article}{
      author={Levi, Friedrich},
      author={van~der Waerden, B.L.},
       title={{\"{U}}ber eine besondere {K}lasse von {G}ruppen},
        date={1933},
     journal={Abh. Math. Semin. Univ. Hambg.},
      volume={9},
       pages={154\ndash 158},
}

\bib{milnor54}{article}{
      author={Milnor, John},
       title={Link groups},
        date={1954},
     journal={Ann. Math.},
      volume={59},
      number={2},
       pages={177\ndash 195},
}

\bib{milnor57}{incollection}{
      author={Milnor, John},
       title={Isotopy of links},
        date={1957},
   booktitle={Algebraic geometry and topology},
   publisher={Princeton University Press},
}

\bib{viazovska17}{article}{
      author={Viazovska, Maryna},
       title={The sphere packing problem in dimension 8},
        date={2017},
     journal={Ann. Math.},
      volume={185},
      number={3},
       pages={991\ndash 1015},
}

\end{biblist}
\end{bibdiv}

\end{document}